\documentclass[11pt]{amsart}

\usepackage{url}
\usepackage{color}
\definecolor{darkgreen}{rgb}{0,0.5,0}
\usepackage[
        colorlinks, citecolor=darkgreen,
        backref,
        pdfauthor={J. Steffen Mueller, C. Stumpe},
        pdftitle={Archimedean local height differences on elliptic curves}
]{hyperref}

\usepackage{amssymb,amsmath, amsthm}
\usepackage{comment}
\usepackage{algorithmic,algorithm}
\usepackage[OT2,OT1]{fontenc}
\usepackage{graphicx}
\usepackage{epsfig}
\usepackage{fancyhdr}
\usepackage{enumerate}

\usepackage{tikz} 
\usetikzlibrary{arrows}

\usepackage[alphabetic,backrefs]{amsrefs} 

\numberwithin{equation}{section}

\newtheorem{thm}{Theorem}[section]

\newtheorem{prop}[thm]{Proposition}

\newtheorem{lemma}[thm]{Lemma}
\newtheorem{cor}[thm]{Corollary}

\theoremstyle{definition}

\theoremstyle{remark}
\newtheorem{rk}[thm]{Remark}
\newtheorem{ex}[thm]{Example}

\newcommand\Q{\mathbb{Q}}
\newcommand\C{\mathbb{C}}

\newcommand\Z{\mathbb{Z}}
\newcommand\R{\mathbb{R}}
\newcommand\Sym{\mathop{\rm Sym}\nolimits}
\newcommand{\BP}{\mathbb{P}}

\begin{document}

\title{Archimedean local height differences on elliptic curves}

\author{J. Steffen M\"uller}
\address{J. Steffen M\"uller,
  Bernoulli Institute, 
  University of Groningen,
  Nijenborgh 9,
  9747 AG Groningen,
  The Netherlands
}
\email{steffen.muller@rug.nl}

\author{Corinna Stumpe}
\address{Corinna Stumpe, Institut f\"ur Mathematik,
          Carl von Ossietzky Universit\"at Oldenburg,
          26111 Oldenburg, Germany}
\email{corinnastumpe@googlemail.com }


\begin{abstract} \setlength{\parskip}{1ex} \setlength{\parindent}{0mm}
  To compute generators for the Mordell-Weil group of an elliptic curve over a number
  field, one needs to bound the difference between the naive and the canonical height from
  above. We give an elementary and fast method to compute an upper bound for the local
  contribution to this difference at an archimedean place, which sometimes gives better
  results than previous algorithms.
\end{abstract}

\maketitle



\section{Introduction}\label{S:intro}
Let $E$ be an elliptic curve defined over a number field $K$. By the Mordell-Weil theorem the 
$K$-rational points on $E$ form a finitely generated group 
\[
  E(K) \cong \Z^r\times \mathrm{Tor}(E(K));
  \]
here $r\ge 0$ is the rank of $E/K$ and
$\mathrm{Tor}(E(K))$ is the (finite) torsion subgroup of $E(K)$.
One of the fundamental computational problems in the study of the arithmetic of elliptic
curves is to compute generators for $E(K)$. Applications of this include, for instance, the numerical
verification of the full conjecture of Birch and Swinnerton-Dyer in examples, as well as
the computation of $S$-integral points on $E$, e.g. using the recent approach of von
K\"anel and Matschke~\cite{vKM}.

Generators of $\mathrm{Tor}(E(K))$ are typically easy to find.
No effective method for the computation of $r$ is known, but there are still several methods which
often succeed in practice. Suppose that we know $r$ and 
$Q_1, \ldots, Q_r \in E(K)$ whose classes generate a finite index subgroup of 
$E(K)/\mathrm{Tor}(E(K))$.
The final step is then to deduce generators of $E(K)$ from this. This is done 
by saturating the lattice generated by $Q_1,\ldots, Q_r$ inside the Euclidean vector space
$(E(K)\otimes \R, \hat{h})$, where $\hat{h}$ is the canonical height.
The most widely used saturation algorithm is due to Siksek~\cite{Sik} and requires, in particular, an
algorithm to enumerate points on $E(K)$ of canonical height bounded by a fixed real number
$B$ (this set is finite by the Northcott property). 

In practice, this is done by first computing an upper bound $\beta$ for the difference between $\hat{h}$ and the
naive height $h:E(K)\to \R$; the points with canonical height bounded by
$B$ are then contained in
\[
  \{P \in E(K) : h(P) \le B+ \beta\},
\]
which can be enumerated for reasonably small $B+\beta$.
Note that the heights we consider are logarithmic, so that $\beta$
shows up exponentially in the size of the search space. It is therefore of great practical importance
to make $\beta$ as small as possible. At the same time, it is desirable to
keep the computation of $\beta$ reasonably fast.

The standard approach for bounding the difference $h-\hat{h}$ is to write it as a sum of
local terms, one for each place of $K$, and to bound the local contributions individually,
see~\cite{CPS} or Section~\ref{S:Diff} below. 
For non-archimedean places, optimal bounds are given
in~\cite{CPS}.
Our main contribution is Theorem~\ref{T:Bound}, which provides an elementary method for bounding the 
local contribution at an archimedean place. This method is 
extremely fast in practice, and yields better results than other existing approaches in many
examples. The approach is analogous to an algorithm due to 
Stoll~\cite{Sto1}, with modifications by 
Stoll and the first-named author \cite{MS2} for Jacobians of 
genus 2 curves and to Stoll~\cite{Sto2} for Jacobians of hyperelliptic genus 3 curves. 
In the case of elliptic curves, the validity of our formulas can be established
using essentially only linear algebra.

This article is partially based on the second-named author's Master thesis~\cite{Stu}.
We thank Michael Stoll for suggesting this project and Peter Bruin for answering several questions about his paper~\cite{Bru} and the
corresponding code.

\section{Action of the two-torsion subgroup}\label{S:2-tors}
In this section, we let $K$ be an algebraically closed field of characteristic zero and 
let $E/K$ be an elliptic curve, given by a Weierstrass equation
\begin{equation}\label{W_eqn}
y^2 + a_1xy + a_3y = x^3 + a_2x^2 + a_4x + a_6\,,
\end{equation}
with point at infinity $O$. We denote by $b_2,\ldots,b_8$ the usual $b$-invariants of
$E$.
Let $\kappa:E \to \BP^1$ be the $x$-coordinate map with respect to the given
equation~\eqref{W_eqn}, extended to all of $E$ by setting
$\kappa(O) = (1:0)$. Given a representative $(x_1,x_2)$ for $\kappa(P)$, 
we have $\kappa(2P) = \delta(x_1,x_2)$, where
$\delta =(\delta_1,\delta_2)$, and 
\begin{align*}
 \delta_1(x_1,x_2)  &= x_1^4 - b_4x_1^2x_2^2 - 2b_6x_1x_2^3 - b_8x_2^4\, ,\\
 \delta_2(x_1,x_2)  &= 4x_1^3x_2 + b_2x_1^2x_2^2 + 2b_4x_1x_2^3 + b_6x_2^4\, .
\end{align*}

The purpose of the present section is to prove an explicit version of the following result.
\begin{prop}\label{P:main}
  There are quadratic forms $y_1,y_2,y_3 \in K[x_1,x_2]$ and constants $a_{ij}, b_{jk}
  \in K$, depending only on $E$, such that for $i=1,2$ and $j =1,2,3$ we have  
  \[
  x_i^2 = \sum^3_{j=1} a_{ij}y_j(x_1,x_2)\;\text{ and }\;y_j(x_1,x_2)^2 = \sum^2_{k=1}
  b_{jk}\delta_k(x_1,x_2)
  \]
  in $K[x_1,x_2]$.
\end{prop}
The constants $a_{ij}$ and $b_{jk}$ are given in~\eqref{aij} and~\eqref{bjk},
respectively.

For \( T \in E[2] \) let \( +_T :E \to E\) be translation by $T$.
Since $-(T+P)= T-P$, the map $+_T$ descends to a map on $\BP^1$. In fact
there is a linear transformation \( \mathfrak{m}_T \) on \( \mathbb{P}^1 \) such that \( \kappa \circ +_T = \mathfrak{m}_T \circ \kappa \). 
A simple calculation shows that $\mathfrak{m}_T$ is represented any non-trivial scalar
multiple of the matrix
\[ M_T := \begin{cases} E_2 &,\; T=O \\ \begin{pmatrix}
x(T) & f'(x(T))-x(T)^2 \\ 1 & -x(T)
\end{pmatrix} &,\; T \not = O \end{cases} \]
where \(f = x^3 + \frac{b_2}{4}x^2+\frac{b_4}{2}x+\frac{b_6}{4} \). 

For the proof of Proposition~\ref{P:main}, 
we analyze the action of \(E[2] \) on the space of homogeneous polynomials in two
variables of degree~2 and~4, respectively.
We first lift the transformation matrices $M_T$ to a
subgroup \( G_E \) of \( \mathrm{SL}_2(K) \) such that \( E[2] \cong G_E / \lbrace \pm E_2 \rbrace \).
Let \( e_2 :E[2]\times E[2] \to \mu_2 \) denote the Weil pairing; then $e_2(T, T') =
\varepsilon(T) \varepsilon(T') \varepsilon(T+T')$, where $\varepsilon(O) :=1$ and
$\varepsilon(T) := -1$ for $T\in E[2] \setminus \{O\}$.
\begin{lemma}\label{M_Tmultiplication}
Let \(
  T, T' \in E[2] \). Then  we have
\[ M_{T'} M_{T} = \begin{cases} \varepsilon(T) \det(M_{T}) E_2 \,,& T=T' \\ e_2(T,
T') M_{T} M_{T'} \,,& T \neq T' . \end{cases} \]
\end{lemma}
\begin{proof}
  The assertion is trivial when $T=O$ or $T'=O$.
Suppose that \( T\ne O \) and \(T'\ne O \).
It is easy to compute
\[ M_{T'} M_{T} 
= \begin{pmatrix}
x(T)x(T') + 2 x(T')^2+\frac{b_2}{2} x(T') + \frac{b_4}{2} & (x(T)-x(T')) (2x(T) x(T') - \frac{b_4}{2}) \\
x(T)-x(T') & 2x(T)^2 + \frac{b_2}{2} x(T) + \frac{b_4}{2} + x(T) x(T')
\end{pmatrix}. \]
In particular, \( M_{T}^2=-\det(M_{T}) E_2 \).
If \(T\) and \(T'\) are distinct, the group law on \(E \) shows
\begin{align*}
 M_{T'} M_{T} 
&= (x(T) - x(T')) \begin{pmatrix}
x(T + T') & 2x(T)x(T') - \frac{b_4}{2} \\ 1 & -x(T + T')
\end{pmatrix} \\
&= (x(T)-x(T')) \begin{pmatrix}
x(T + T') & f'(x(T+T')) - x(T + T')^2 \\ 1 & -x(T + T')
\end{pmatrix}
= - M_{T} M_{T'},
\end{align*}
which proves the result.
\end{proof}

Lemma~\ref{M_Tmultiplication} shows that the classes of the matrices $M_T$ form a subgroup of \(
\mathrm{PSL}_2(K) \). We now lift this to a subgroup of 
\( \mathrm{SL}_2(K) \).

\begin{lemma}\label{G_E}
For \( T \in E[2] \) let \( \gamma_T \in K^\times \) such that \( \gamma_T^2= \det(M_T)^{-1}
  \) and let \( \tilde{M}_T:= \gamma_T M_T \). Then
\[ G_E := \lbrace \pm \tilde{M}_T \, \vert \, T \in E[2] \rbrace \]
is a subgroup of \( \mathrm{SL}_2(K) \). Moreover, \( G_E \) is isomorphic to the
  quaternion group $Q_8$, and
\[ E[2] \cong G_E / \lbrace \pm E_2 \rbrace .\]
\end{lemma}
\begin{proof}
Obviously  \( G_E \subset \mathrm{SL}_2(K) \) and \(G_E \) does not depend on the
choice of \( \gamma_T \). 
By Lemma \ref{M_Tmultiplication}, we have $M^{-1} \in G_E$ for $M \in G_E$.
Let \( T_1,T_2,T_3 \in E[2] \) be nontrivial and pairwise distinct.
Since
\[ \kappa \circ +_{T_3}
= \kappa \circ +_{T_1} \circ +_{T_2} 
= \mathfrak{m}_{T_1} \circ \kappa \circ +_{T_2}
= \mathfrak{m}_{T_1} \circ \mathfrak{m}_{T_2} \circ \kappa, \]
we have
\[ M_{T_1} M_{T_2} = \gamma M_{T_3} \]
for a unit \( \gamma \in K^\times \). 
From \( \det(\tilde{M}_{T_1} \tilde{M}_{T_2}) = \det(\gamma_{T_1} \gamma_{T_2} \gamma
  M_{T_3}) =1 \) we deduce that \( \gamma_{T_3} \) is 
equal to \( \gamma_{T_1} \gamma_{T_2} \gamma \) up to sign, so \(
  \tilde{M}_{T_1} \tilde{M}_{T_2} G_E \), and hence $G_E$ is indeed a subgroup of
  $\mathrm{SL}_2(K)$.

The remaining statements are clear.
\end{proof}

\begin{proof}[Proof of Proposition~\ref{P:main}]
Let $\rho$ denote the standard representation 
  of \( G_E \) on the vector space \(V \) of $K$-linear forms in $x_1, x_2$.
Then the symmetric square $\rho^2$ factors through $E[2]$. Hence we can view  \( \rho^2 \)
as a representation of \( E[2] \) on \( \Sym^2(V) \), and we have
\[ \rho^2= \bigoplus_{T \in E[2] \setminus \lbrace O \rbrace} e_2(\cdot, T). \]
It is easy to check that for each nontrivial \(2\)-torsion point \(T \) the polynomial 
\[ y_T :=  x_1^2 - 2 x(T) x_1 x_2 - (f'(x(T))-x(T)^2) x_2^2 \in \Sym^2(V) \]
is an eigenform of \( \rho^2 \). 
Fix any ordering of the non-trivial 2-torsion points and call them $T_1,T_2,T_3$; let $y_j
  := Y_{T_j}$.
  Since \( \mathcal{Y} := (y_1,y_2,y_3)
\) is linearly independent, \( \mathcal{Y} \) forms a basis for \( \Sym^2(V) \).
We find that the coefficients of $x_1^2$ and $x_2^2$ with respect to $\mathcal{Y}$ are
given by
\[ I_\mathcal{Y}(x_1^2)
= \tau^{-1} \begin{pmatrix}
f'(x(T_1))-x(T_1)^2 \\ f'(x(T_2))-x(T_2)^2 \\ f'(x(T_3))-x(T_3)^2
\end{pmatrix}
\times \begin{pmatrix}
x(T_1) \\ x(T_2) \\ x(T_3)
\end{pmatrix}\]
and 
\[ I_\mathcal{Y}(x_2^2)
= \tau^{-1} \begin{pmatrix}
1 \\ 1 \\ 1
\end{pmatrix} 
\times 
\begin{pmatrix}
x(T_1) \\ x(T_2) \\ x(T_3)
\end{pmatrix},
\]
where \( \tau:=\sum_{i} (f'( x( \sigma^i (T_1))) -x( \sigma^i (T_1))^2) (x(\sigma^i(T_2))
- x(\sigma^i(T_3)) ) \neq 0 \) for the cycle \( \sigma = (T_1 \, T_2 \, T_3)\).
In other words, we have $x_i^2 = \sum^3_{j=1} a_{ij}y_j(x_1,x_2)$ for $i =1,2$, where 
\begin{align}\label{aij}
& \begin{pmatrix}
a_{11} & a_{12} & a_{13} \\ a_{21} & a_{22} & a_{23}
\end{pmatrix} \nonumber \\
= & \begin{pmatrix}
\frac{2x(T_2)x(T_3)-\frac{b_4}{2}}{2(x(T_1)-x(T_2))(x(T_1)-x(T_3))}
& \frac{2x(T_1)x(T_3)-\frac{b_4}{2}}{2(x(T_2)-x(T_1))(x(T_2)-x(T_3))}
& \frac{2x(T_1)x(T_2)-\frac{b_4}{2}}{2(x(T_3)-x(T_1))(x(T_3)-x(T_2))} 
\\
\frac{-1}{2(x(T_1)-x(T_2))(x(T_1)-x(T_3))}
& \frac{-1}{2(x(T_2)-x(T_1))(x(T_2)-x(T_3))}
& \frac{-1}{2(x(T_3)-x(T_1))(x(T_3)-x(T_2))}
\end{pmatrix}.
\end{align}
As for $\rho^2$, we have that \( \rho^4 \) factors through \( E[2] \). 
Since projectively \( \delta(\tilde{M}_T (x_1,x_2))=\delta(x_1,x_2) \),
the polynomials $\delta_1, \delta_2$ are \( E[2] \)-invariant under the
fourth symmetric power \( \rho^4 \). Moreover, they are linearly independent. As the space of $E[2]$-invariant quartic
polynomials is 2-dimensional, it is spanned by $\delta_1$ and $\delta_2$.
Computing the squares \( y_j^2\), we find that
$y_j(x_1,x_2)^2 = \sum^2_{k=1} b_{jk}\delta_k(x_1,x_2)$, where 
\begin{equation}\label{bjk}
 \begin{pmatrix}
b_{11} & b_{12} \\ b_{21} & b_{22} \\ b_{31} & b_{32}
\end{pmatrix} 
= \begin{pmatrix}
1 & -x(T_1) \\ 1 & -x(T_2) \\ 1 & -x(T_3)
\end{pmatrix}.
\end{equation}
This completes the proof of Proposition~\ref{P:main}.
\end{proof}

\section{Global height differences}\label{S:Diff}
Let $K$ be a number field. We define $M_K$ to be the set of places of $K$, where we
normalize the absolute value $|{\cdot}|_v$ associated to $v\in M_K$ by requiring that it
extends the usual absolute value on~$\Q$ when $v$ is an infinite place and by setting $|p|_v = p^{-1}$
when $v$ is a finite place above a prime number~$p$. 
For $v \in M_K$, we set $n_v = [K_v : \Q_w]$ 
where $w$ is the place of~$\Q$ below~$v$. Then the product formula
$\prod_{v \in M_K} |x|_v^{n_v} = 1$ holds for all $x \in K^\times$.

Consider an elliptic curve $E/K$, given by an integral Weierstrass equation~\eqref{W_eqn}.
We define the {\em naive  height} of $P \in E(K)\setminus\{O\}$ by
\[
  h(P) =  h(\kappa(P)) = \frac{1}{[K:\Q]}\sum_{v \in M_K}n_v \log\max\{|x_1|_v,|x_2|_v\}\,,
\]
where $(x_1, x_2)\in K^2$ represents $\kappa(P)$. The {\em canonical height} of $P$ is defined as the limit
\[ \hat{h}(P) = \lim_{n \to \infty} 4^{-n}{h(2^nP)}. \]
In this work, we are not really interested in the canonical height itself, but rather in
upper bounds on the difference $h-\hat{h}$. As in~\cite{CPS} and~\cite{MS1}, we
decompose the difference into a finite sum of local terms
\[
 h(P) -\hat{h}(P) = \frac{1}{[K:\Q]}\sum_{v \in M_K}n_v\Psi_v(P)\,,
\]
where the functions $\Psi_v:E(K_v) \to \R$ are continuous and bounded. It is
then clear that it suffices to compute upper bounds on all $\Psi_v$ to deduce an upper bound
on the difference $h-\hat{h}$. 
Recall that we may write
\[
  \Psi_v(P) = -\sum^{\infty}_{n=0} 4^{-n-1} \log\Phi_v(2^nP)
\]
for $P \in E(K_v)$, where $\Phi:E(K_v)\to \R$ is the continuous bounded function defined by
\[
  \Phi_v(P) := \frac{\max\{|\delta_1(x_1,x_2)|_v,
  |\delta_2(x_1,x_2)|_v\}}{\max\{|x_1|_v,|x_2|_v\}^4}
\]
and $(x_1,x_2) \in K_v^2$ represents $\kappa(P)$. See~\cite{CPS} and~\cite{MS1} for 
details.

\section{Archimedean local height differences}\label{S:Bounds}
In this section we show how to bound the local contribution $\Psi_v$ to the height
difference, where $v$ is an archimedean place of a number field.
We will drop $v$ from the notation for simplicity and assume that $K_v =\C$,
unless stated otherwise.
So consider an elliptic curve $E/\C$, given by a Weierstrass equation~\eqref{W_eqn}.
Note that Proposition~\ref{P:main} lets us bound
$|x_i|^4$ ($i = 1,2$) in terms of $|\delta_j(x_1,x_2)|$, $j = 1,2$. From this 
we easily get an upper bound for $\Phi$ using the triangle inequality. Via the geometric
series we deduce:
\begin{cor}\label{C:first_bound}
  We have
  \[
\max_{P \in E(\C)} \lbrace \Psi(P) \rbrace 
\leq 
  \frac{4}{3} \left( \sqrt{\sum_{j=1}^3 |a_{ij}| \sqrt{|b_{j1}|   +|b_{j2}| } }
  \right)_{i=1,2},
  \]
where the constants $a_{ij},\, b_{jk} \in \C$ are as in Proposition~\ref{P:main}.
\end{cor}

This idea was first used by Stoll~\cite{Sto1, Sto2} to bound the height difference for Jacobians
of genus 2 curves and hyperelliptic genus 3 curves, respectively. We will follow his approach closely; in fact,
the elliptic case is much simpler. 
Furthermore, we will iterate the bound for $\Phi$ to get a
better bound for $\Psi$ than the one obtained from the geometric series; this was used by
Stoll and the first-named author for genus~2~\cite{MS2}, and by Stoll~\cite{Sto2} for genus 3.

For the iteration we define the function
\begin{align*}
\varphi: \mathbb{R}_{\geq 0}^2 \to \mathbb{R}_{\geq 0}^2, (d_1,d_2) \mapsto \left(
  \sqrt{\sum_{j=1}^3 |a_{ij}| \sqrt{|b_{j1}| \, d_1 +|b_{j2}| \, d_2} } \right)_{i=1,2}
\end{align*}
and we set
\[
  c_N :=\frac{4^N}{4^N-1} \log(\| \varphi^{\circ N}(1,1)\|)
  \]
for 
$N \ge 1$, where $\|\cdot\|$ denotes the supremum norm. Hence $c_1$ is precisely the upper
bound from Corollary~\ref{C:first_bound}.
Our algorithm for bounding $\Psi$ is based on the following result, whose statement and
proof follow \cite[Lemma 16.1]{MS2}.
\begin{thm}\label{T:Bound}
  The sequence \( (c_N)_{N \ge 1} \) is monotonically decreasing and we have 
\begin{align*}
\max_{P \in E(\C)} \lbrace \Psi(P) \rbrace
\leq c_N
\end{align*}
  for every \(N \ge 1\).
\end{thm}
\begin{proof} 
To verify that the upper bound holds, let $\alpha \in \C^2$. A simple induction shows that 
for $N\ge 1$ we have
  \[
| \alpha_i | \leq \varphi^{\circ N}\left( |{\delta}^{\circ N}(\alpha)_1|, \, |
    {\delta}^{\circ N}(\alpha)_2 | \right)_i,
  \]
  and since
  \[
| {\delta}^{\circ N}(\alpha)_i |
\leq
 \varphi(1,1)_i \, \| {\delta}^{\circ (N+1)}(\alpha) \|^{\frac{1}{4}},
  \]
we find that 
\begin{align*}
| \alpha_i |
\leq 
 \varphi^{\circ N} \big( \| {\delta}^{\circ (N+1)}(\alpha) \|^{\frac{1}{4}} \, \varphi(1,1) \big)_i.
\end{align*}
  Shifting $N$ by 1 and using that $\varphi$ is homogeneous of degree~$1/4$, it follows that 
  \begin{equation}\label{alpha_bd}
|\alpha_i| \leq \| {\delta}^{\circ N}(\alpha) \|^{\frac{1}{4^{N}}} \, \varphi^{\circ N}(1,1)_i.
  \end{equation}
  We now apply~\eqref{alpha_bd} to $\alpha = 
  {\delta}^{\circ Nn}(x_1,x_2)$, where $n\ge 1$ and $x \in \C^2$ represents $\kappa(P)$ for $P \in
  E(\C)$, to obtain
\begin{align*}
  \| {\delta}^{\circ Nn}(x_1,x_2) \|
\leq  \| {\delta}^{\circ N(n+1)}(x ) \|^{\frac{1}{4^N}} \, \| \varphi^{\circ N}(1,1) \|.
\end{align*}
Upon noting that
  \[
\Psi(P)
  =  \sum_{n=0}^\infty 4^{-Nn} \log \left( \frac{\| {\delta}^{\circ Nn}(x_1,x_2) \|}{\|
  {\delta}^{\circ N(n+1)}(x_1,x_2) \|^{\frac{1}{4^N}}} \right)
  \]
the result follows. 

To show that $c_N$ is monotonically decreasing, consider the function
\begin{align*}
\psi: \mathbb{R}^2 \to \mathbb{R}^2, \alpha \mapsto \big( \log(\varphi(\exp(\alpha_1),
  \, \exp(\alpha_2))_i) \big)_{i=1,2}.
\end{align*}
Note that the Jacobi matrix of $\psi$ has positive entries and that its rows sum to $1/4$, because $\varphi_1$ and
$\varphi_2$ are homogeneous of degree $1/4$. 
It follows that 
\begin{equation}\label{contraction}
  \|\psi(\alpha)-\psi(\beta)\| \le \frac14 \| \alpha-\beta\|
\end{equation}
for all $\alpha, \beta \in \C^2$.

For $N\ge 1$ we have $c_N = \frac{4^N}{4^N-1}\|\psi^{\circ N}(0,0)\|$.
In particular,~\eqref{contraction} implies
\[\|\psi^{\circ 2}(0,0) - \psi(0,0)\| \le \frac 14\| \psi(0,0)\|,\] whence
$c_2 \le c_1$.
We now proceed by induction on $N$; so let $N\ge 2$ such that $c_{N} \le c_{N-1}$.
  Let $b_N := \|\psi^{\circ N}(0,0)\|$. If $b_{N+1} \le b_N$, then we're done, so we may
  assume that $b_{N+1}>b_N$.
  Applying~\eqref{contraction} to $\alpha = \psi^{\circ N} (0,0)$ and $\beta = \psi^{\circ
  (N-1)} (0,0)$, we find that
  \[
    b_{N+1} \le \min\left\{\frac54b_N -\frac14 b_{N-1}, \frac34b_N + \frac14 b_{N-1}\right\},\]
  according to whether $b_{N}\ge b_{N-1}$ or not.
  First assume that $b_N\ge b_{N-1}$, so that $$b_{N+1} \le \frac54b_N -\frac14 b_{N-1}.$$
  From $c_N \le c_{N-1}$ we get 
  \[ -\frac14b_{N-1} \le -\frac{4^{N-1}-1}{4^N-1}b_N, \]
  which implies 
  \[
    b_{N+1} \le \left(\frac54 - \frac{4^{N-1}-1}{4^N-1}\right) b_N
    =\frac{4^{N+1}-2}{4^{N+1}-4}b_N < \frac{4^{N+1}-1}{4^{N+1}-4}b_N ,
    \]
  and hence $c_{N+1} < c_N$. The case $b_N < b_{N-1}$ is similar.
\end{proof}
In particular, $(c_N)_N$ and $(b_N)_N$ both converge to the same limit, and this limit is an
upper bound for $\Psi$. In practice, the sequence converges quickly, and a few iterations
suffice. This gives us a very simple method to bound $\Psi$ from above. 
\begin{rk}\label{R:Real}
  Suppose that $v$ is a real place and that $E(\R)$ has only one component. Then $b_{22}$ and $b_{32}$ are non-real, but
  all $P \in E(K_v)$ have real coordinates, so we have 
  \begin{align*}
    |y_j(x_1,x_2)^2| &= \left|\sum^2_{k=1} b_{jk}\delta_k(x_1,x_2)\right|\\& \le \max \left\{
    |b_{j1}\delta_1(x_1,x_2) +
  b_{j2}\delta_2(x_1,x_2)|, |b_{j1}\delta_1(x_1,x_2) - b_{j2}\delta_2(x_1,x_2)|\right\}
  \end{align*}
    for $j \in \{2,3\}$ and $x\in \R^2$ representing $\kappa(P)$. Modifying the definition
    of the function $\varphi$ accordingly, we often get a better bound in practice.
\end{rk}

\section{Alternative algorithms}\label{S:other}
In this section we briefly discuss other approaches to bounding $\Psi_v$ from above for an
archimedean place $v$.
The approach of   Cremona-Prickett-Siksek \cite{CPS} is to find the largest value
$\gamma$ of $\Phi_v$; then $\gamma/3$ is an upper bound for $\Psi_v$.
For real places this translates into a simple algorithm which is trivial to
implement. For complex places, they give two approaches: one based on Gr\"obner bases and
another one based on refining an initial crude bound via repeated quadrisection. 
The latter is faster and yields better bounds in practice than the former. 
The method of \cite{CPS} is implemented in {\tt Magma} and as part of Cremona's {\tt
mwrank} (which is also contained in {\tt Sage}).
A variation of this approach was presented by Uchida~\cite{Uch}; he computes the
largest value of an analogue of $\Phi_v$, but with duplication replaced by multiplication by
$m$ for $m>2$.

An alternative approach is to use that for $K_v=\C$, which we may
assume without loss of generality, $\Psi_v$ can be expressed in terms of the Weierstrass
$\wp$-function and  an
archimedean canonical local height function, which in turn is closely related
to the Weierstrass $\sigma$-function. 
This was used by Silverman~\cite{Sil2} to provide an easily computed upper bound for
$\Psi_v$ in terms of the values of the $j$-invariant and the discriminant of $E$;
according to~\cite{CPS}, this bound is usually larger than the one due to
Cremona-Prickett-Siksek, at least for real embeddings.
In a spirit similar to the repeated quadrisection method in~\cite{CPS}, Bruin~\cite{Bru}
uses a recursive approach (starting from a
fundamental domain of the period lattice of $E/\C$) to approximate the maximal value taken
by $\Psi_v$ on $E(\C)$ to any desired precision. Bruin's algorithm therefore gives nearly
optimal bounds for complex embeddings, whereas for real embeddings the bound computed using
the algorithm of Cremona-Prickett-Siksek is often smaller. 
A {\tt Pari/GP} implementation of Bruin's method can be found at
{\url{https://www.math.leidenuniv.nl/~pbruin/hdiff.gp}} (note that this uses a different
normalization from ours; the difference is $\log|\Delta|_v/6$, where $\Delta$ is the
discriminant of the given Weierstrass model).
While this method is reasonably fast for curves with small coefficients,
it can be slow even for
medium-sized coefficients. For instance, it took about 18 minutes to compute an upper bound for the curve
with Cremona label 11a2, which has minimal Weierstrass equation 
\[
  y^2 + y = x^3 - x^2 - 7820x - 263580.
\]
So while this approach leads to superior bounds, it is somewhat less useful in practice,
because the need for computing a very sharp upper bound mostly arises for curves whose
coefficients are relatively large. 

\section{Experiments and comparison}\label{S:Ex}
We implemented an algorithm based on Theorem~\ref{T:Bound} and Remark~\ref{R:Real} to compute an upper bound for $\Psi_v$ for an archimedean place $v$ 
in {\tt Magma}~\cite{Magma}. The code is available at {\url{https://github.com/steffenmueller/arch-ht-diff}}.
We experimentally compared our code with the {\tt Magma}-implementation of the algorithm of~\cite{CPS}, using a single core on an Intel Xeon(R) CPU E3-1275 V2 3.50GHz processor.
Note that the latter sometimes shows that the upper bound is exactly 0 (which is attained by
$P=O$), whereas our code always returns a positive real number.
We compared all curves of conductor at most 35.000 in Cremona's database of
elliptic curves over the rationals. Here and in the following $\beta$ is the upper bound
returned by our code and $\beta_{CPS}$ is the upper bound returned by the Magma
implementation of the algorithm from~\cite{CPS}.  We also list the average value of $\beta$ and
$\beta_{CPS}$ (including the cases where the latter is 0).

\begin{center}
    \begin{tabular}{|c|c|c|c|c|c|} \hline
      max. conductor & $\beta_{CPS}=0$&  $\beta>\beta_{CPS}$& $\beta<\beta_{CPS}$ & avg.
      $\beta_{CPS}$ & avg. $\beta$ \\
\hline
      $10.000$ & $33.5\%$ & $38.8 \%$ & $27.8\%$ & $0.947$& $0.992$ \\
      $20.000$ & $33.7\%$ & $37.9 \%$ & $28.3\%$ & $0.979$& $1.007$ \\
      $35.000$ & $33.8\%$& $37.5 \%$ & $28.8\%$ & $1.001$& $1.007$ \\
      \hline
    \end{tabular}
\end{center}

    We found similar results for databases of `small' elliptic curves over real quadratic fields.

    Perhaps surprisingly, the picture is quite different for `random' curves, and it
    would be interesting to investigate why this is the case. The
    following table contains the respective results for $10^5$ randomly chosen elliptic curves over $\Q$
    with Weierstrass coefficients $a_1,\ldots,a_6$ bounded in absolute value by 
    $B\in \{10^2,10^3,10^4\}$.

\begin{center}
    \begin{tabular}{|c|c|c|c|c|c|c|} \hline
      $B$ & $\beta_{CPS}=0$&  $\beta>\beta_{CPS}$& $\beta<\beta_{CPS}$ & avg.
      $\beta_{CPS}$ & avg. $\beta$ \\
\hline
      $10^2$ &$46.3\%$ & $3.3\%$& $50.4\%$& $0.145$ & $0.045$\\
      $10^3$ &$48.4\%$ & $1.0\%$& $50.7\%$& $0.146$ & $0.011$\\
      $10^4$ &$49.2\%$ & $0.3\%$& $50.5\%$& $0.147$ & $0.002$\\
      \hline
    \end{tabular}
\end{center}

    In the above tables, the average time it took to compute the bounds was very
    short (less than 0.002 seconds on average) for both algorithms.
    A comparison over $\Q(\sqrt{5})$ with the same parameters resulted in the following:
    
\begin{center}
    \begin{tabular}{|c|c|c|c|c|c|c|} \hline
      $B$ & $\beta_{CPS}=0$&  $\beta>\beta_{CPS}$& $\beta<\beta_{CPS}$ & avg.
      $\beta_{CPS}$ & avg. $\beta$ \\
\hline
      $10^2$ &$21.4\%$ & $6.8\%$& $71.9\%$& $0.148$ & $0.039$\\ 
      $10^3$ &$23.5\%$ & $1.9\%$& $74.6\%$& $0.148$ & $0.010$\\
      $10^4$ &$24.6\%$ & $0.4\%$& $75.0\%$& $0.150$ & $0.002$\\
      \hline
    \end{tabular}
\end{center}

    So it seems that for large coefficients, our algorithm yields better results most of
    the time, unless $\beta_{CPS}=0$. We also see that, on average, our bound is much
    smaller. Here is a particularly striking example.

    \begin{ex}
      Let $E/\Q$ be given by
      \begin{align*}
        y^2 &+ xy + y = x^3 - x^2 + 31368015812338065133318565292206590792820353345x
        \\&+302038802698566087335643188429543498624522041683874493555186062568159847
      \end{align*}
      This example was found by Elkies in 2009 and currently holds the record for the elliptic
      curve of largest known rank ($r=19$) which is provably correct, independently of
      any conjectures. In this case $\beta_{CPS} = 18.018$, whereas $\beta =0.147$.
    \end{ex}

We also compared the two implementations for a few thousand curves with coefficient sizes
as above, but over some imaginary
quadratic fields. Here we found that our bound was better in all examples. Moreover, it
also took less time to compute in all cases (on average 0.003 seconds compared to 1.2 seconds).

In practice, the methods of this paper, of Cremona-Prickett-Siksek and of Bruin should
be combined. For a real embedding, one should first compute an upper bound using
Cremona-Prickett-Siksek. If this is non-zero, one should then apply our algorithm, and use
whichever bound is smaller. For a complex embedding, our algorithm appears to be a good
first choice.
If the resulting bound seems too large for saturation, and if the
coefficients of the curve are of reasonable size, one can then compute a bound using the
algorithm of Bruin. This is basically optimal for complex embeddings,
and sometimes beats the other bounds for real embeddings as well, but, as discussed above,
it typically takes much longer to compute.

\end{document}